\documentclass[12pt]{amsart}
\usepackage{amsfonts}
\usepackage{amssymb}
\usepackage{bm}

 \newtheorem{theorem}{Theorem}[section]
 
 \newtheorem{conjecture}[theorem]{Conjecture}
 \newtheorem{lemma}[theorem]{Lemma}
 \newtheorem{proposition}[theorem]{Proposition}
 \theoremstyle{definition}

 \newtheorem{remark}[theorem]{Remark}

\numberwithin{equation}{section}

\renewcommand{\b}{\hm\cdot}

\begin{document}


\baselineskip=17pt


\title[On the invariant $\mathsf E(G)$ for groups of odd order]{On the invariant $\mathsf E(G)$ for groups of odd order}

\author[W.D. Gao]{Weidong Gao}
\address{Center for Applied Mathematics \\ Tianjin University\\
Tianjin, 300072, P.R.
China}
\email{wdgao@nankai.edu.cn, weidong.gao@tju.edu.cn, wdgao1963@aliyun.com}

\author[Y.L. Li]{Yuanlin Li*}
\address{Department of Mathematics and Statistics\\  Brock University\\
St. Catharines, ON
L2S 3A1, Canada}
\email{yli@brock.ca}

\author[Y.K. Qu]{Yongke Qu}
\address{Department of Mathematics\\ Luoyang Normal University\\
LuoYang, 471934, P.R.
China}
\email{yongke1239@163.com}

\thanks{*Corresponding author: Yuanlin Li, E-mail: yli@brocku.ca}

\date{}

\begin{abstract}
Let $G$ be a multiplicatively written finite group. We denote by $\mathsf E(G)$ the smallest integer $t$ such that every sequence of $t$ elements in $G$ contains a product-one subsequence of length $|G|$. In 1961, Erd\H{o}s, Ginzburg and Ziv proved that $\mathsf E(G)\leq 2|G|-1$ for every finite cyclic group $G$ and this result is well known as the Erd\H{o}s-Ginzburg-Ziv Theorem. In 2010, Gao and Li proved that  $\mathsf E(G)\leq\frac{7|G|}{4}-1$ for every finite non-cyclic solvable group and they conjectured that $\mathsf E(G)\leq \frac{3|G|}{2}$ holds for any finite non-cyclic group. In this paper, we confirm the conjecture for all finite non-cyclic groups of odd order.
\end{abstract}

\subjclass[2020]{Primary 11B75; Secondary 11P70}

\keywords{Erd\H{o}s-Ginzburg-Ziv theorem, Product-one sequence, Finite solvable group, Davenport constant}

\maketitle

\section{Introduction and main results}
Let $G$ be a finite group written multiplicatively. Let $S=g_1\bm\cdot \ldots\bm\cdot g_{\ell}$ be a sequence over $G$ with length $\ell$. We use $$\pi(S)=\{g_{\tau (1)}\ldots g_{\tau (\ell)}: \tau \mbox { a permutation of } [1,\ell]\}\subseteq G$$ to denote the set of products of $S$. We say that $S$ is a {\sl product-one sequence} if $1\in\pi(S)$. Let $\mathsf d(G)$ be the {\sl small Davenport constant} of $G$ (i.e., the maximal integer $\ell$ such that there is a sequence of length $\ell$ over $G$ which has no nontrivial product-one subsequence). We denote by $\mathsf E(G)$ (resp. $\mathsf s(G)$) the smallest integer $t$ such that every sequence of $t$ elements in $G$ has a product-one subsequence of length $|G|$ (resp. $\mbox{exp}(G)=\mbox{lcm}\{\mbox{ord}(g)|g\in G\}$). When $G$ is an abelian group, Gao \cite{G1996} discovered  the essential relation between $\mathsf E(G)$ and $\mathsf d(G)$, i.e., $\mathsf E(G)=\mathsf d(G)+|G|$. This result inspired much research and it was generalized in various directions. As a striking example, we mention that the entire Chapter 16 of Grynkiewicz's monograph \cite{Gr}
is devoted to a generalization towards weighted zero-sum sequences (over abelian groups). This relation also holds for some non-abelian groups such as dihedral or dicyclic groups (see \cite{Bass2007,GL2008,Han2015,HZ2019,ZG2005} for details). Recently in \cite{QL}, the above relation is verified for some special metacyclic  groups $C_p\ltimes C_m$ (which are not necessarily dihedral groups). Also, Oh and Zhong \cite{OZ} investigated the associated inverse problem on $\mathsf E(G)$ and $\mathsf s(G)$ for dihedral and  dicyclic groups.  We note that there are connections of the combinatorial study of sequences and their invariants (such as $\mathsf E(G)$, $\mathsf d(G)$ and $\mathsf s(G)$), with factorization theory and with invariant theory. The interested reader is referred to some recent publications \cite{CDG, CDS, GGOZ, Gr}.

In 1961, Erd\H{o}s, Ginzburg and Ziv \cite{EGZ1961} proved that $\mathsf E(G)\leq 2|G|-1$ for every finite cyclic group $G$ and this result is well known as the Erd\H{o}s-Ginzburg-Ziv Theorem. The authors of \cite{EGZ1961} also remarked that their proof works for finite abelian groups, i.e., $\mathsf E(G)\leq 2|G|-1$. In 1976 \cite{O}, Olson proved that the same result holds for any finite group. In 1984, Yuster and Peterson \cite{YP1984} showed that $\mathsf E(G)\leq 2|G|-2$ when $G$ is a non-cyclic solvable group. Later Yuster \cite{YP1988} improved the result to $\mathsf E(G)\leq 2|G|-r$ provided that $|G|\geq 600((r-1)!)^2$. In 1996, Gao \cite{Gao1996} further improved the upper bound to $\mathsf E(G)\leq \frac{11|G|}{6}-1$. Recently in 2010 Gao and Li \cite{GL2010} proved that $\mathsf E(G)\leq\frac{7|G|}{4}-1$ and they proposed the following conjecture regarding the best possible upper bound for $\mathsf E(G)$.

\begin{conjecture}\label{mainconj}
Let $G$ be a finite non-cyclic group. Then $\mathsf E(G)\leq\frac{3|G|}{2}$.
\end{conjecture}

In this paper, we are able to confirm Conjecture \ref{mainconj} for non-cyclic groups of odd order and our main result is as follows.

\begin{theorem}\label{maintheorem} Let $G$ be a non-cyclic group of odd order $|G|>9$. Then $\mathsf E(G)\leq \frac{3|G|-3}{2}$.
\end{theorem}

\begin{remark}\label{C_3^2} The minimal (in order) non-cyclic group of odd order is $C_3\oplus C_3$ and we have $\mathsf E(C_3\oplus C_3)=13=\frac{3|G|-1}{2}$ (see \cite[Corollary 1 (ii)]{G1996}).
\end{remark}

\section{Notation and Preliminaries}
We use the notation and conventions described in detail in \cite{GG2013}.

For real numbers $a,b\in \mathbb{R}$, we set $[a,b]=\{x\in \mathbb{Z}: a\leq x\leq b\}$. For integers $m,n\in \mathbb{Z}$, we denote by $\mbox{gcd}(m,n)$ the greatest common divisor of $m$ and $n$.

Let $G$ be a finite multiplicative group. For a prime divisor $p$ of $|G|$, denote by $G_p$ a Sylow $p$-subgroup of $G$, and by $G_{p'}$ a maximal $p'$-subgroup of $G$ such that $p$ is not a divisor of $|G_{p'}|$ but the only prime divisor of $|G|/|G_{p'}|$. If $A \subseteq G$ is a nonempty subset, then denote by  $\langle A \rangle $  the subgroup of $G$ generated by $A$. If $A$ and $B$ are subsets of $G$, we define the product-set as $AB=\{ab:a\in A,b\in B\}$. Recall that by a {\sl sequence} over a group $G$ we mean a finite, unordered sequence where the repetition of elements is allowed. We view sequences over $G$ as elements of the free abelian monoid $\mathcal{F}(G)$, denote multiplication in $\mathcal{F}(G)$ by the bold symbol $\bm\cdot$ rather than by juxtaposition, and use brackets for all exponentiation in $\mathcal{F}(G)$.

A sequence $S \in \mathcal F(G)$ can be written in the form $S= g_1  \bm \cdot g_2 \bm \cdot \ldots \bm\cdot g_{\ell},$ where $|S|= \ell$ is the {\it length} of $S$. For $g \in G$, let $\mathsf v_g(S) = |\{ i\in [1, \ell] : g_i =g \}|\,  $ denote the {\it multiplicity} of $g$ in $S$. A sequence $T \in \mathcal F(G)$ is called a {\it subsequence } of $S$ and is denoted by $T \mid S$ if  $\mathsf v_g(T) \le \mathsf v_g(S)$ for all $g\in G$. Denote by $S \bm\cdot T^ {[-1]}$  the subsequence of $S$ obtained by removing the terms of $T$ from $S$. Let $A$ be a subset of $G$. Denote by $S_A$ the subsequence of $S$ obtained by all the terms of $A$ from $S$.

If $S_1, S_2 \in \mathcal F(G)$, then $S_1 \bm\cdot S_2 \in \mathcal F(G)$ denotes the sequence satisfying that $\mathsf v_g(S_1 \bm\cdot S_2) = \mathsf v_g(S_1 ) + \mathsf v_g( S_2)$ for all $g \in G$. For convenience we  write
\begin{center}
 $g^{[k]} = \underbrace{g \bm\cdot \ldots \bm\cdot g}_{k} \in \mathcal F(G)\quad$
\end{center}
for $g \in G$ and $k \in \mathbb{N}_0$.

Suppose $S= g_1 \b g_2 \bm \cdot \ldots \bm\cdot g_{\ell} \in \mathcal F(G)$. Let $$\pi (S) = \{g_{\tau(1)}\ldots g_{\tau(\ell)}: \tau \mbox{ a permutation of } [1, \ell] \} \subseteq G$$  denote the {\it set of products} of $S$. Let $$\Pi_n(S) = \cup _{T\mid S,\ |T| = n}\pi(T)$$ denote the {\it set of all $n$-products} of $S$.
Let $$\Pi(S) = \cup_{1 \le n \le \ell}\Pi_n(S)$$ denote the {\it set of all subsequence products} of $S$. The sequence $S$ is called
\begin{itemize}
\item[$\bullet$]  {\it product-one} if $1 \in \pi(S)$;
\item[$\bullet$] {\it product-one free} if $1\not\in \Pi(S)$;
\item[$\bullet$] {\it minimal product-one } if $1\in\pi(S)$ and $S$ has no proper product-one subsequence.
\end{itemize}

If $\mathbf{A}=(A_1,\ldots, A_m)$ is a sequence of finite subsets of $G$, let $k\leq m$, we define $$\Pi^{k}(\mathbf{A})=\{a_{i_1}\ldots a_{i_{k}}:1\leq i_1<\cdots<i_{k}\leq m\mbox{ and } a_{i_j}\in A_{i_j} \mbox{ for every } 1\leq j\leq k\}.$$  Let $A$ be a subset of $G$ and $\mbox{stab}(A)=\{g\in G: gA=A\}$ its stabilizer. The following lemma is a generalization of Kneser's Theorem which is crucial for our proof of the main result.

\begin{lemma}\cite[Theorem 1.3]{DGM2009}\cite[Theorem 13.1]{Gr}\label{genKneser}
Let $\mathbf{A}=(A_1,\ldots, A_{\ell})$ be a sequence of finite subsets of an abelian group $G$, let $k\leq {\ell}$, and let $H=\mbox{stab}(\Pi^{k}(\mathbf{A}))$. If $\Pi^{k}(\mathbf{A})$ is nonempty, then $$|\Pi^{k}(\mathbf{A})|\geq |H|\bigg(1-k+\sum_{Q\in G/H}min\big\{k,|\{i\in[1,{\ell}]:A_i\cap Q\neq \emptyset\}|\big\}\bigg).$$
\end{lemma}

\begin{lemma}(Theorem of Erd\H{o}s-Ginzburg-Ziv)\cite{EGZ1961}\label{EGZ} Let $G$ be a finite cyclic group. Then $\mathsf E(G)=2|G|-1$.
\end{lemma}

\begin{lemma}\cite[Lemma 5]{GL2010}\label{induction} Let $1<c\leq 2$ be a constant. Let $H$ be a normal subgroup of a finite group $G$. If $\mathsf E(H)\leq c|H|-1$, then $\mathsf E(G)\leq c|G|-1$.
\end{lemma}

\begin{lemma}\cite[Theorem 7.5]{GG2006}\label{inverse} Let $G$ be cyclic of order $n\geq 2$, $k\in[2, \lfloor n/4\rfloor +2]$ and $S\in \mathcal F(G)$ be a sequence of length $|S|=2n-k$. If $S$ has no product-one subsequence of length $n$, then $$S=a^{[u]}\bm\cdot b^{[v]}\bm\cdot c_1\bm\cdot \ldots \bm\cdot c_{\ell}, \mbox{ where } ord(ab^{-1})=n, u\geq v\geq n-2k+3 \mbox{ and }$$  $u+v\geq 2n-2k+1$ (equivalently, $\ell\leq k-1$). In particular, we have
\begin{itemize}
\item If $k=2$, then $S=a^{[n-1]}\bm\cdot b^{[n-1]}$.
\item If $k=3$ and $n\geq 4$, then $S=a^{[n-1]}\bm\cdot b^{[n-2]}$ or $S=a^{[n-1]}\bm\cdot b^{[n-3]}\bm\cdot (b^2a^{-1})$.
\end{itemize}
\end{lemma}

\begin{lemma}\label{known}
Let $G$ be a finite non-cyclic group of odd order and $G\not\cong C_3^2$. Then $\mathsf E(G)\leq \frac{3|G|}{2}-1$ in the following cases:
\begin{itemize}
\item[(i)] $G$ is nilpotent;
\item[(ii)] $G$ has a normal subgroup $N$ such that $G/N\cong C_m\ltimes C_m$;
\item[(iii)] $G=\langle x, y| x^{p}=1=y^m, x^{-1}yx=y^r\rangle,$ where $p$ is the smallest prime divisor of $G$ and $\mbox{gcd}(p(r-1),m)=1$.
\end{itemize}
\end{lemma}
\begin{proof} For the groups in this lemma, proofs of $\mathsf E(G)\leq |G|+\frac{|G|}{p}+p-2$ can be found in \cite{Han2015,HZ2019,QL}, where $p$ is the smallest prime divisor of $G$. Since $G$ is of odd order and $G\not\cong C_3^2$, we have $|G|+\frac{|G|}{p}+p-2\leq \frac{3|G|}{2}-1$.
\end{proof}

\begin{lemma}\cite[Corollary 10.5.2]{Ha}\label{supersolvable}
Let $G$ be a finite supersolvable group and $p$ the smallest prime divisor of $|G|$. Then there exists a normal subgroup $H$ of index $p$.
\end{lemma}

\section{Proof of Theorem \ref{maintheorem}}
Since every group of odd order is solvable, in what follows, we always assume that $G$ is solvable. Since $G$ is non-cyclic of odd order $>9$, we need only consider the group $G$ with $|G|\geq 21$. By using the minimal counterexample method we will prove that $\mathsf E(G)\leq \frac{3|G|-3}{2}$, or equivalently, $\mathsf E(G)\leq \frac{3|G|}{2}-1$ since $|G|$ is odd. Throughout this section, we always assume that $G$ is a minimal counterexample (i.e., $G$ is a non-cyclic group of minimal order $|G|\geq 21$ such that $\mathsf E(G)\geq\frac{3|G|-1}{2}$) with the smallest prime divisor $p\geq 3$ of $|G|$.

\begin{lemma}\label{counterexample}
Let $G$ be a minimal counterexample. Then $$G=\langle x, y| x^{p^{\alpha}}=1=y^m, x^{-1}yx=y^r\rangle,$$ where $r^p\equiv 1\pmod m$, $\mbox{gcd}(r-1,m)=1$ and $p|q-1$ for every prime $q|m$.
\end{lemma}

\begin{proof}
In terms of Lemma~\ref{induction} with $c=\frac{3}{2}$ and Remark~\ref{C_3^2}, we obtain that every proper normal subgroup of $G$ is cyclic or isomorphic to $C_3^2$. Since $G$ is of odd order, $G$ is solvable. Thus $G$ has a proper normal subgroup $G_0$ of prime index. We distinguish the proof into the following two cases.
\\

\noindent {\bf Case 1}  Every proper normal subgroup of $G$ is cyclic.

Then $G_0$ is cyclic. Therefore, every subgroup of $G_0$ is a normal subgroup of $G$. We conclude that $G$ is supersolvable. By Lemma~\ref{supersolvable}, there exists a normal subgroup $H$ such that $|G/H|=p$. By Lemma~\ref{known}~(i), $G$ is not a $p$-group. Suppose $q||G|$ and $q\neq p$ is a prime. Since $H$ is a normal subgroup, $H$ is cyclic. We conclude that every subgroup of $H$ is a normal subgroup of $G$. Thus $G_q=H_q\vartriangleleft G$.

Let $N\vartriangleleft H$ be the subgroup of $H$ with $q=|H|/|N|\neq p$ a prime. Then $N\vartriangleleft G$. We claim that $G/N$ is not abelian. Assume to the contrary that $G/N\cong C_{pq}$. Then there exists a normal subgroup $K$ such that $G/K\cong (G/N)/(K/N)\cong C_q$. Thus $K_p=G_p$. Since $K\vartriangleleft G$, $K$ is cyclic. Thus $G_p\vartriangleleft G$. So, every sylow subgroup of $G$ is normal. Therefore, $G$ is a nilpotent group, yielding a contradiction to Lemma~\ref{known}~(i). Thus $G/N$ is not abelian. Therefore $G/N\cong C_p\ltimes C_q$ for every subgroup $N\vartriangleleft H$ with $q=|H|/|N|\neq p$ a prime. Then $p|q-1$.

It's clear that $G_{p'}=H_{p'}\vartriangleleft G$ and $G/G_{p'}$ is a $p$-group. If $G/G_{p'}$ is not cyclic, then there exists a normal subgroup $K\supseteq G_{p'}$ such that $G/K\cong (G/G_{p'})/(K/G_{p'})\cong C_p\oplus C_p$, yielding a contradiction to Lemma~\ref{known}~(ii). Thus $G/G_{p'}$ is cyclic. We have $$G=\langle x, y| x^{p^{\alpha}}=1=y^m, x^{-1}yx=y^r\rangle\cong C_{p^{\alpha}}\ltimes C_m,$$ where $p|q-1$ for every prime $q|m$. Since $G/H\cong C_p$, we have $x^{p}\in H$. Then $y^{r^p}=x^{-p}yx^{p}=y$.  Thus $r^p\equiv 1\pmod m$.

Notice that $\langle x, y^{r-1}\rangle\vartriangleleft G$. If $\langle x, y^{r-1}\rangle\neq G$, then $\langle x, y^{r-1}\rangle$ is cyclic. Since $\langle x\rangle <\langle x, y^{r-1}\rangle$, we have $G_p=\langle x\rangle \vartriangleleft G$. Therefore, $G$ is nilpotent, yielding a contradiction. So $\langle x, y^{r-1}\rangle=G$. Thus $\langle y^{r-1}\rangle=\langle y\rangle$. We have $\mbox{gcd}(r-1,m)=1$.
\\

\noindent {\bf Case 2}  There exists a normal subgroup $N\cong C_3^2$.

If $G_0$ is not cyclic, then $G_0=N$. Therefore, $|G|=9q$ where $q=|G/G_0|$. If $q=3$ or $G_q\vartriangleleft G$, then $G$ is nilpotent, yielding a contradiction to Lemma~\ref{known}~(i). Assume that $q\geq 5$ and $G_q\not\vartriangleleft G$. Then the number of Sylow $q$-subgroup $n_q|9$ and $n_q\neq 1$. Since $n_q\equiv 1\mod q$, we have $q=2$, yielding a contradiction to $2\nmid |G|$.

If $G_0$ is cyclic, then $G/G_0\cong C_3$. As in Case 1, we have $G/G_{3'}$ is cyclic, yielding a contradiction to $G_{3'}\cap N=\{1\}$.
\end{proof}

Let $G=\langle x, y| x^{p^{\alpha}}=y^m=1, x^{-1}yx=y^r\rangle\cong C_{p^{\alpha}}\ltimes C_m$, where $p$ is the smallest prime divisor of $|G|$, and $\mbox{gcd}(p(r-1), m)=1$. Let $K=\langle x\rangle$, $N=\langle y\rangle$ and $H=\langle x^p, y\rangle$. Then $C_{m}\cong N\lhd G$, $K\cong G/N\cong C_{p^{\alpha}}$ and $ C_{mp^{\alpha-1}}\cong H\lhd G$. Let $\varphi$ be the canonical homomorphism from $G$ onto $G/N$. Then for each sequence $T$ over $G$, $\varphi(T)$ is a sequence over $G/N$. The following lemma was proved in a recently submitted paper \cite{QL}. We include a proof here for the convenience of the reader.

\begin{lemma}\label{basic} Let $M$ be any subgroup of $N=\langle y\rangle$, $u$ be an element of $N$ and $0\leq s<s'\leq p-1$. Then
\begin{itemize}
\item[(i)] If $u^{r^s}\in M$, then $u\in M$.
\item[(ii)] If both $u^{r^s}$ and $u^{r^{s'}}$ are in the same coset of $M$, then $u\in M$.
\item[(iii)] If $u\neq 1$, then $u^{r^s}\neq u^{r^{s'}}$.
\end{itemize}
\end{lemma}
\begin{proof} (i) Since $\mbox{gcd}(r,m)=1$, we have $\mbox{gcd}(r^s, m)=1$. Therefore, there exist $k, e\in \mathbb{Z}$ such that $1=r^sk+me$.  From $u\in N$ we have $u^m=u^{|N|}=1$. If $u^{r^s}\in M$, then $u=u^{r^sk+me}=(u^{r^s})^k\in M$.

(ii) Since both $u^{r^s}$ and $u^{r^{s'}}$ are in the same coset of $M$, $u^{r^s-r^{s'}}\in M$, and thus $u^{(1-r^{s'-s})r^s}\in M$. Since $u^{1-r^{s'-s}}\in N$, by (i) we conclude that $u^{1-r^{s'-s}}\in M$. Next we show that $\mbox{gcd}(1-r^{s'-s},m)=1$. Assume to the contrary that $\mbox{gcd}(1-r^{s'-s},m)\neq 1$. Then there exists a prime divisor $q$ of $m$ such that $1-r^{s'-s}\equiv 0 \pmod q$.  Since $r^p\equiv 1\pmod m$, we have $r^p\equiv 1\pmod q$, which together with  $r^{s'-s}\equiv 1 \pmod q$ gives $r^{(s'-s,p)} \equiv 1\pmod q$. But $\mbox{gcd}(s'-s, p)=1$. Therefore, $r\equiv 1 \pmod q$, a contradiction to $\mbox{gcd}(r-1,m)=1$. Hence, we must have $\mbox{gcd}(1-r^{s'-s},m)=1$. In a similar way to (i), we get $u\in M$.

(iii) Assume to the contrary that $u^{r^s}=u^{r^{s'}}$. Then both $u^{r^s}$ and $u^{r^{s'}}$ are in the same coset of the trivial subgroup $\{1\}$. By (ii), $u\in \{1\}$, so $u=1$, yielding a contradiction.
\end{proof}

The following lemma is a generalization of \cite[Lemma 16]{Bass2007} and we include a proof for the reader's convenience. Let $H_i=x^iH$ be the $i$th coset of $H$ in $G$ for $0\leq i\leq p-1$.

\begin{lemma}\label{conjugationtwo}
Let $T_0$ be a sequence over $G$ such that $\varphi (T_0)$ is a product-one sequence over $G/N$. Suppose $T_0'|T_0$ with length $|T_0'|=p$ and $T_0'|S_{H_k}$ for some $1\leq k\leq p-1$. For every $j\in[1, \ell]$, let $T_j$ be a sequence over $G$ such that $\pi(T_j)\cap N\neq \emptyset$, and let $u_j\in \pi(T_j)\cap N$. Then, for every $t\in [1, \ell]$, $\pi(T_0\bm\cdot T_1\bm\cdot\ldots\bm\cdot T_{t})$ contains the product set $\{u_0\}\{u_1,u_1^{r},\ldots,u_1^{r^{p-1}}\}\ldots \{u_t,u_t^{r},\ldots,u_t^{r^{p-1}}\}$ for some $u_0\in \pi(T_0)$.
\end{lemma}

\begin{proof}
By rearranging the order of sequence $T_0$, we may write $T_0=T_0'\bm\cdot T_0''$. Let $T_0'= x^kh_1\bm\cdot \ldots \bm\cdot x^kh_p$ where $h_i\in H$ for all $i\in[1,p]$. We first consider the sequence $T_0\bm\cdot T_1$. Let $u_0\in x^kh_1 \ldots x^kh_p\pi(T_0'')\subseteq \pi(T_0)$. Since $u_1\in \pi(T_1)$, we have $x^kh_1\ldots x^kh_iu_1x^kh_{i+1}\ldots x^kh_p\pi(T_0'')\subseteq \pi(T_0\bm\cdot T_1)$ for every $i\in [1,p-1]$ and $x^kh_1\ldots  x^kh_p u_1\pi(T_0'')\subseteq \pi(T_0\bm\cdot T_1)$. For every $i\in [1,p-1]$, we have
$$x^kh_1\ldots x^kh_iu_1x^kh_{i+1}\ldots x^kh_p\pi(T_0'')=x^kh_1\ldots x^kh_pu_1^{r^{k(p-i)}}\pi(T_0''),$$
as $u_1\in N$. Note that $\varphi (T_0)$ is a product-one sequence over $G/N$, $|T_0'|=p$ and $T_0'|S_{H_k}$ for some $1\leq k\leq p-1$. We have $\pi(T_0'')\subseteq H$, and thus
$$x^kh_1\ldots x^kh_pu_1^{r^{k(p-i)}}\pi(T_0'')=x^kh_1\ldots x^kh_p\pi(T_0'')u_1^{r^{k(p-i)}},$$
for every $i\in [1,p-1]$ and $x^kh_1\ldots  x^kh_p u_1\pi(T_0'')=x^kh_1\ldots  x^kh_p \pi(T_0'')u_1\subseteq \pi(T_0\bm\cdot T_1)$. Therefore,
$$\pi(T_0\bm\cdot T_1)\supseteq\{u_0\}\{u_1,u_1^{r^k},\ldots,u_1^{r^{k(p-1)}}\}.$$
Since $k\in [1,p-1]$, we have $\{k, 2k, \ldots, (p-1)k\}\equiv \{1, 2, \ldots, p-1\} \pmod p$. So,
$$\pi(T_0\bm\cdot T_1)\supseteq\{u_0\}\{u_1,u_1^{r},\ldots,u_1^{r^{p-1}}\}.$$

By induction and by repeating the above argument, we can easily prove the general result.
\end{proof}

We are now in position to prove our main result.
\\

\noindent {\bf Proof of Theorem \ref{maintheorem}.}

Let $G$ be a minimal counterexample. By Lemma \ref{counterexample}, $G=\langle x, y| x^{p^{\alpha}}=1=y^m, x^{-1}yx=y^r\rangle$ where $r^p\equiv 1\pmod m$, $\mbox{gcd}(r-1,m)=1$ and $p|q-1$ for every prime $q|m$. Let $K=\langle x\rangle\cong G/N\cong C_{p^{\alpha}}$, $N=\langle y\rangle\cong C_m$ and $H=\langle x^p, y\rangle\cong C_{mp^{\alpha-1}}$. Let $\varphi$ be the canonical homomorphism from $G$ onto $G/N$. By Lemma~\ref{known}~(iii), we have $\alpha\geq 2$. We will show that $\mathsf E(G)\leq \frac{3mp^{\alpha}-3}{2}\leq \frac{3mp^{\alpha}}{2}-1$, yielding a contradiction. Thus such a minimal counterexample does not exist, proving Theorem \ref{maintheorem}. Let $S$ be any sequence over $G$ with length $\frac{3mp^{\alpha}-3}{2}$. We need to prove that $1\in \Pi_{mp^{\alpha}}(S)$. We distinguish the proof into the following two cases.
\\

\noindent {\bf Case 1} $|S\bm\cdot S_H^{[-1]}|\geq p^{\alpha}+(p-1)^2$.

Let $S_0|S_H$ be a subsequence of $S_H$ with length $|S_0|=\max\{0, 2p^{\alpha}-1-|S\bm \cdot S_H^{[-1]}\}|\leq \max\{0, p^{\alpha}-(p-1)^2-1\}$. Since $\varphi(S_0\bm \cdot S\bm \cdot S_H^{[-1]})$ is a sequence over $G/N\cong C_{p^{\alpha}}$ with length $\geq 2p^{\alpha}-1$, by Lemma~\ref{EGZ}, there exists a subsequence $T_0|S_0\bm \cdot S\bm\cdot S_H^{[-1]}$ with $|T_0|=p^{\alpha}$ and $\pi(T_0)\subseteq N$. Let $v$ be the maximal integer such that $T_0\bm\cdot T_1\bm\cdot\ldots\bm\cdot T_v|S$, $\varphi(T_j)$ is a product-one subsequence with $|T_j|=p^{\alpha}$ and $\pi(T_j)\neq \{1\}$ where $j\in[1,v]$. Let $w$ be the maximal integer such that $T_0\bm\cdot T_1\bm\cdot\ldots\bm\cdot T_{v+w}|S$, $\varphi(T_j)$ is a product-one subsequence with $|T_j|=p^{\alpha}$ and $\pi(T_j)=\{1\}$ where $j\in [v+1,v+w]$. Since $G/N\cong C_{p^{\alpha}}$, by the maximality of $v$ and $w$, we have $|S\bm\cdot (T_0\bm\cdot T_1\bm\cdot\ldots\bm\cdot T_{\ell})^{[-1]}|=|\varphi(S)|-|\varphi(T_0\bm\cdot T_1\bm\cdot\ldots\bm\cdot T_{\ell})|\leq 2p^{\alpha}-2$. Therefore, $p^{\alpha}\ell=|\varphi(T_0\bm\cdot T_1\bm\cdot\ldots\bm\cdot T_{\ell})|-p^{\alpha}\geq |S|-2p^{\alpha}+2-p^{\alpha}\geq \frac{3mp^{\alpha}+1}{2}-3p^{\alpha}$, and thus $\ell\geq \frac{3m-5}{2}$. Let $1\neq u_j\in\pi(T_j)$ for all $j\in [1,v]$ and $1=u_j\in\pi(T_j)$ for $j\in [v+1,v+w]$. Since $|S_0|\leq \max\{0, p^{\alpha}-(p-1)^2-1\}$, $|T_0|=p^{\alpha}$ and $G/H\cong C_p$, there exists a subsequence $T_0'|T_0$ such that $T_0'|S_{H_k}$ with $|T_0'|=p$ for some $1\leq k\leq p-1$. By Lemma~\ref{conjugationtwo}, $$\pi(T_0\bm\cdot T_{i_1}\bm\cdot\ldots\bm\cdot T_{i_{m-1}})\supseteq \{u_0\}\{u_{i_1},u_{i_1}^{r},\ldots,u_{i_1}^{r^{p-1}}\}\ldots\{u_{i_{m-1}},u_{i_{m-1}}^{r},\ldots,u_{i_{m-1}}^{r^{p-1}}\}$$ for some $u_0\in \pi(T_0)$ and every $(m-1)$-subset $\{i_1,\ldots, i_{m-1}\}\subset[1,\ell]$. Let $A_0=\{u_0\}$, $A_j=\{u_{j},u_{j}^{r},\ldots,u_{j}^{r^{p-1}}\}$ for $j\in [1,v]$, and $A_j=\{1\}$ for $j\in [v+1,\ell]$. It's clear that $A_j\subseteq N$ for $j\in [0,\ell]$. Let $\mathbf{A}=(A_1,\ldots, A_{\ell})$. Notice that $\ell\geq \frac{3m-5}{2}\geq m-1$. Then $$\Pi_{mp^{\alpha}}(S)\supseteq A_0\Pi^{m-1}(\mathbf{A}).$$ Let $M=\mbox{stab}(\Pi^{m-1}(\mathbf{A}))$. By Lemma~\ref{genKneser}, $$|\Pi^{m-1}(\mathbf{A})|\geq |M|\Big(2-m+\sum_{Q\in N/M}\min\big\{m-1,|\{j\in[1,\ell]: A_j\cap Q\neq \emptyset\}|\big\}\Big).$$
Let $I_M$ be the subset of $[1,\ell]$ such that $j\in I_M$ if and only if $A_j\subseteq M$. Since $M$ is a subgroup of $N$, if $|I_M|\geq (m/|M|) |M|+|M|-1=m+|M|-1$, then by using Lemma~\ref{EGZ} repeatedly for $m/|M|$ times, we can find a subset $\{j_1,\ldots, j_m\}\subseteq I_M$ such that $u_{j_1}\ldots u_{j_m}=1$. Thus $1\in \pi(T_{j_1}\bm\cdot \ldots\bm\cdot T_{j_m})$. Since $|T_{j_1}\bm\cdot \ldots\bm\cdot T_{j_m}|=mp^{\alpha}$, $1\in\Pi_{mp^{\alpha}}(S)$ as desired.

Next, we always assume that $|I_M|\leq m+|M|-2$ and we show that $|\Pi^{m-1}(\mathbf{A})|\geq m$. Let $Q\in N/M$ and $V_Q=\{j\in[1,\ell]: A_j\cap Q\neq \emptyset\}$. Clearly, $V_M\supseteq I_M$. By Lemma~\ref{basic}~(i), $V_M\subseteq I_M$. Thus $V_M=I_M$. Moreover, if $Q\neq M$, then $V_Q\cap V_M=V_Q\cap I_M=\emptyset$. By Lemma~\ref{basic}~(ii), if $j\in V_Q$ for some $Q\neq M$, then all the elements of $A_j$ are in $p$ different cosets of $M$, i.e., if $j\notin I_M$, then $A_j\cap M= \emptyset$ and $$|\{Q\in N/M: A_j\cap Q\neq \emptyset\}|=p.\ \ \ \ \ \ \ \ \ \ \ \ \ \ \ (*)$$  Let $$\mu=|\{Q\in N/M: |V_Q|\geq m\}|.$$ If $\mu=0$, then
\begin{align*}
|\Pi^{m-1}(\mathbf{A})|\geq &|M|\Big(2-m+\sum_{Q\in N/M}\min\big\{m-1,|\{j\in[1,\ell]: A_j\cap Q\neq \emptyset\}|\big\}\Big)\\
=&|M|\Big(2-m+|V_M|+\sum_{Q\in N/M, Q\neq M}|\{j\in[1,\ell]: A_j\cap Q\neq \emptyset\}|\Big)\\
=&|M|\Big(2-m+|I_M|+\sum_{Q\in N/M, Q\neq M}\sum_{j\in[1,\ell], A_j\cap Q\neq \emptyset}1\Big)\\
=&|M|\Big(2-m+|I_M|+\sum_{j\in[1,\ell]\setminus I_M}\sum_{Q\in N/M, A_j\cap Q\neq \emptyset}1\Big)\\
=&|M|(2-m+|I_M|+p(\ell-|I_M|)) \ \ \ \ \ \ \mbox{ by } (*)\\
=& |M|(2-m+p\ell-(p-1)|I_M|)\\
\geq &|M|(2-m+\frac{3mp-5p}{2}-(p-1)(m-1))\\
&(\mbox{as } \ell\geq \frac{3m-5}{2} \mbox{ and } |I_M|\leq m-1)\\
=& |M|(\frac{mp-3p}{2}+1)\\
=& \frac{mp-3p}{2}+1\\
\geq&m \ \ \ \ \ \ \ \ (\mbox{as } p\geq 3 \mbox{ and } m\geq 2p+1\geq 7).
\end{align*}

If $\mu\geq 2$, then
\begin{align*}
|\Pi^{m-1}(\mathbf{A})|\geq |M|(2-m+2(m-1))\geq m.
\end{align*}

Finally, we consider the case when $\mu=1$. Let $R\in N/M$ be the unique coset of $M$ such that $|V_R|\geq m$. Assume that $R\neq M$. Then $R=\alpha M$ for some $\alpha\in N\setminus M$. Let $\alpha_j\in A_j \cap R=A_j \cap \alpha M$ for all $j\in V_R$. Since $\alpha_j\in A_j=\{u_j, u_j^{r}\ldots, u_j^{r^{p-1}}\}$, $\alpha_j^{r}\in A_j$. Thus $\alpha_j^{r}\in A_j\cap \alpha^{r}M$ for all $j\in V_R$. This implies that for all $j\in V_R$, $j\in V_{\alpha^{r}M}$ and thus $|V_{\alpha^rM}|\geq |V_R|$. Since $\alpha\notin M$, by Lemma~\ref{basic}~(ii), $\alpha M\neq \alpha^r M$. So we have found another coset $\alpha^r M(\neq R)$ such that $|V_{\alpha^rM}|\geq |V_R|\geq m$, yielding a contradiction to $\mu=1$. Thus we must have $R=M$. Therefore,
\begin{align*}
|\Pi^{m-1}(\mathbf{A})|\geq &|M|\Big(2-m+\sum_{Q\in N/M}\min\big\{m-1,|\{j\in[1,\ell]: A_j\cap Q\neq \emptyset\}|\big\}\Big)\\
=&|M|\Big(2-m+m-1+\sum_{Q\in N/M, Q\neq M}|\{j\in[1,\ell]: A_j\cap Q\neq \emptyset\}|\Big)\\
=&|M|\Big(1+\sum_{Q\in N/M, Q\neq M}\sum_{j\in[1,\ell], A_j\cap Q\neq \emptyset}1\Big)\\
=&|M|\Big(1+\sum_{j\in[1,\ell]\setminus I_M}\sum_{Q\in N/M, A_j\cap Q\neq \emptyset}1\Big)\\
=& |M|(1+p(\ell-|I_M|))\ \ \ \ \ \ \mbox{ by } (*)\\
\geq& |M|(1+\frac{3mp-5p}{2}-p(m+|M|-2))\\
&(\mbox{as } \ell\geq \frac{3m-5}{2} \mbox{ and } |I_M|\leq m+|M|-2)\\
=& |M|(\frac{mp-p}{2}-p|M|+1)\\
\geq& (\frac{|M|p}{2}-1)(m-2|M|-1)-|M|-1+m \\
\geq& (\frac{p}{2}-1)(7|M|-2|M|-1)-|M|-1+m \\
& (\mbox{as } |M|\geq 1 \mbox{ and } \frac{m}{|M|}\geq 2p+1\geq 7) \\
\geq& m \ \ \ \ (\mbox{as } p\geq 3).
\end{align*}
So in all the cases, we have shown $|\Pi^{m-1}(\mathbf{A})|\geq m$. Thus $\Pi^{m-1}(\mathbf{A})\supseteq N$. Therefore, $\Pi_{mp^{\alpha}}(S)\supseteq A_0\Pi^{m-1}(\mathbf{A})\supseteq N \ni 1$. This completes the proof of Case 1.
\\

\noindent {\bf Case 2} $|S\bm\cdot S_H^{[-1]}|\leq p^{\alpha}+(p-1)^2-1$.

Note that if $G\neq C_{3^{2}}\ltimes C_7$, then $|S_H|=|S|-|S\bm\cdot S_H^{[-1]}|\geq \frac{3mp^{\alpha}-3}{2}-(p^{\alpha}+(p-1)^2-1)=mp^{\alpha}+(\frac{mp^{\alpha}-1}{2}-p^{\alpha}-(p-1)^2)\geq mp^{\alpha}+mp^{\alpha-1}-1$. By Lemma~\ref{EGZ}, there exist $p$ disjoint product-one subsequences $T_1, \ldots, T_p$ of $S_H$ with length $|T_j|=mp^{\alpha-1}$ for all $j\in [1,p]$. Therefore, $T_1\bm\cdot\ldots\bm\cdot T_p$ is a product-one subsequence of $S$ with length $mp^{\alpha}$.

Finally, we consider the case that $G=C_{3^{2}}\ltimes C_7$. Let $S$ be a sequence over $G$ with length $\frac{3|G|-3}{2}=93$. We need only to show there exists a product-one subsequence with length $|G|=63$ of $S$. If $|S_H|\geq |G|+|H|-1=63+21-1=83$, then as above we are done. Now assume that

$$|S_H|\leq 82.$$

Assume to the contrary that there does not exist any product-one subsequence with length $|G|=63$ of $S$. We distinguish the rest of the proof into the following two subcases.
\\

\noindent {\bf Subcase 2.1} $|S_H|=82$.

By using Lemma~\ref{EGZ} twice, we obtain $S_H=T_1\bm\cdot T_2\bm\cdot T$, where $\pi(T_1)=\pi(T_2)=1$ with length $|T_1|=|T_2|=21$, and $1\notin \Pi_{21}(T)$ with length $|T|=40$. Let $T=h_1\bm\cdot \ldots \bm\cdot h_{40}$ and $\mathbf{A}=(\{h_1\},\ldots, \{h_{40}\})$, then $\Pi_{21}(T)=\Pi^{21}(\mathbf{A})$. Let $M=\mbox{stab}(\Pi^{21}(\mathbf{A}))$. If $M\neq 1$, then $40\geq 21+|H/M|-1$. There exists a subsequence $T'|T$ with $|T'|=21$ and $\pi(T')\in M$. Thus $\Pi^{21}(\mathbf{A})\supseteq M\ni 1$, yielding a contradiction to $1\notin \Pi_{21}(T)$. Thus $M=1$. By Lemma~\ref{genKneser}, $\Pi_{21}(T)=\Pi^{21}(\mathbf{A})=H\setminus\{1\}$. Then $\Pi_{21}(h\bm\cdot T)=H$, where $h|T_2$.

Notice that $|S\bm\cdot S_H^{[-1]}|=11\geq |G/H|=3$. There exists a nonempty subsequence $S'|S\bm\cdot S_H^{[-1]}$ such that $\pi(S')\subseteq H$. Let $T_2'|T_2\bm\cdot h^{[-1]}\bm\cdot S'$ with $S'|T_2'$ and $|T_2'|=21$, then $\pi(T_2')\subseteq H$. Since $\Pi_{21}(h\bm\cdot T)=H$, there exists a subsequence $T_3|h\bm\cdot T$ with length $|T_3|=21$ such that $1\in \pi(T_2')\pi(T_3)$. So $1\in \pi(T_1\bm\cdot T_2'\bm\cdot T_3)$, yielding a contradiction since $|T_1\bm\cdot T_2'\bm\cdot T_3|=63$. This completes the proof of Subcase 2.1.
\\

\noindent {\bf Subcase 2.2} $|S_H|\leq 81$.

Now we have $93=|S|=|S_H|+|SS_H^{[-1]}|\leq 81+(3^2+(3-1)^2-1)=93$. This forces that
$$|S\bm\cdot S_H^{[-1]}|=12 \mbox{ and } |S_H|=81.$$

Let $S_0|S_H$ be a subsequence of $S_H$ with length $|S_0|=4$. Then $\varphi(S\bm\cdot S_H^{[-1]}\bm\cdot S_0)$ is a sequence with length $16=18-2$ over $G/N\cong C_{3^2}$. Notice that $\varphi(g)\neq \varphi(h)$ for any $g|S\bm\cdot S_H^{[-1]}$ and $h|S_0$. Therefore, $\varphi(S\bm\cdot S_H^{[-1]}\bm\cdot S_0)$ is not of the form $a^{[8]}\bm\cdot b^{[8]}$ where $a,b\in G/N$. By Lemma~\ref{inverse}, there exists a subsequence $W|S_0\bm \cdot S\bm \cdot S_H^{[-1]}$ such that $\varphi(W)$ is a product-one subsequence over $\varphi(G)$ with length $|W|=9$. Since $|S_0|=4$, $|W|=9$ and $|G/H|=3$, there exists a subsequence $W'|W$ such that $W'|S_{H_k}$ with $|W'|=3$ for some $1\leq k\leq 2$.  As in Case 1, we have $1\in \Pi_{63}(S)$, yielding a contradiction. This completes the proof of Subcase 2.2.\qed

\section{Concluding Remarks}
We recall some results on $\mathsf E(G)$ for finite non-cyclic groups of even order.
\begin{proposition}\cite[Theorems 8 and 10]{Bass2007}\label{even}
Let $G$ be a non-cyclic group of even order. Then $\mathsf E(G)=\frac{3|G|}{2}$ in the following cases.
\begin{itemize}
\item[(i)] $G=D_{2m}$ is the dihedral group of order $2m$;
\item[(ii)] $G=Q_{4m}$ is the dicyclic group of order $4m$.
\end{itemize}
\end{proposition}

It follows from the above proposition that for the given dihedral and dicyclic groups, we have $\mathsf E(G)=|G|+\mathsf d(G)$, and we note that the associated inverse problem was solved in \cite[Theorems 1.2 and 1.3]{OZ}. It is known that $\mathsf E(G)\geq |G|+\mathsf d(G)$ for all finite groups
(see \cite[Lemma 2.2]{OZ} for details), and  we are not aware of any finite group $G$ with $\mathsf E(G)>|G|+\mathsf d(G)$.

The above proposition also shows that $\frac{3|G|}{2}$ is the best possible upper bound for $\mathsf E(G)$. By Lemma \ref{induction}, we can obtain the following description of a minimal counterexample $G$ (i.e., $G$ is a non-cyclic group of minimal even order such that $\mathsf E(G)\geq\frac{3|G|}{2}$+1).

\begin{proposition}\label{evencounterexample} Let $G$ be a minimal counterexample and $\{1\}=H_0\vartriangleleft H_1\vartriangleleft\cdots\vartriangleleft H_n=G$ be a subnormal series of $G$ with $H_{n-1}\neq H_n$. If $H_k\not\cong C_3^2$, then $H_k$ is cyclic or $2||H_k|$ with $\mathsf E(H_k)=\frac{3|H_k|}{2}$ for $1\leq k\leq n-1$.
\end{proposition}
\begin{proof} Since $G$ is a minimal counterexample, if $H_k$ is not cyclic, then $\mathsf E(H_k)\leq \frac{3|H_k|}{2}$ for all $1\leq k\leq n-1$. Assume to the contrary that $H_k$ is not cyclic, and either $2\nmid |H_k|$ or $2||H_k|$ with $\mathsf E(H_k)\leq \frac{3|H_k|}{2}-1$. In the former case, by Theorem \ref{maintheorem}, we have $\mathsf E(H_k)\leq\frac{3|H_k|}{2}-1$ which is also true for the latter case. Now by using Lemma \ref{induction} repeatly, we obtain $\mathsf E(H_i)\leq\frac{3|H_i|}{2}-1$ for all $k\leq i\leq n$. Especially, we have $\mathsf E(G)\leq\frac{3|G|}{2}-1$, yielding a contradiction.
\end{proof}

We remark that determining the structure of the group $G$ described in Proposition \ref{evencounterexample} will be very helpful to completely confirm Conjecture \ref{mainconj}.

\subsection*{Acknowledgements}
We would like to thank the referee for their valuable suggestions which helped improve the readability and the presentation of the paper. This work was carried out during a visit of the third author to Brock University as an international visiting scholar. He would like to sincerely thank the host institution for its hospitality and for providing an excellent atmosphere for research. This work was supported in part by the National Science Foundation of China (No. 11701256, 11871258, 12071344), the Youth Backbone Teacher Foundation of Henan's University (No. 2019GGJS196), the China Scholarship Council (Grant No. 201908410132), and it was also supported in part by a Discovery Grant from the Natural Sciences and Engineering Research Council of Canada (Grant No. RGPIN 2017-03903).


\normalsize


\begin{thebibliography}{[HD82]}




\normalsize
\baselineskip=17pt


\bibitem{Bass2007} J. Bass, \emph{ Improving the Erd\H{o}s-Ginzburg-Ziv theorem for some non-abelian groups}, J. Number Thoery 126 (2007), 217--236.

\bibitem{CDG} K. Cziszter, M. Domokos and A. Geroldinger, \emph{The interplay of invariant theory with multiplicative ideal theory and with arithmetic combinatorics}, in Multiplicative Ideal Theory and Factorization Theory, Springer, 2016, pp. 43--95.

\bibitem{CDS} K. Cziszter, M. Domokos and I. Sz\"{o}ll\H{o}si, \emph{The Noether number and the Davenport constants of the groups of order less than $32$}, J. Algebra 510 (2018), 513--541.

\bibitem{DGM2009} M. DeVos, L. Goddyn and B. Mohar, \emph{A generalization of Kneser's Addition Theorem}, Adv. Math. 220 (2009), 1531--1548.

\bibitem{EGZ1961} P. Erd\H{o}s, A. Ginzburg and A. Ziv, \emph{Theorem in the additive number theory}, Bull. Res. Council Israel 10F (1961), 41--43.

\bibitem{G1996} W. Gao, \emph{A combinatorial problem on finite abelian groups}, J. Number Theory 58 (1996), 100--103.

\bibitem{Gao1996} W. Gao, \emph{An improvement of Erd\H{o}s-Ginzburg-Ziv theorem}, Acta Math. Sinca 39 (1996), 514--523.

\bibitem{GG2006} W. Gao and A. Geroldinger, \emph{Zero-sum problems in finite abelian groups: A survey}, Expo. Math.  24 (2006), 337--369.

\bibitem{GL2010} W. Gao and Y. Li, \emph{The Erd\H{o}s-Ginzburg-Ziv theorem for finite solvable groups}, J. Pure Appl. Algebra 214 (2010), 898--909.

\bibitem{GL2008} W. Gao and Z. Lu, \emph{The Erd\H{o}s-Ginzburg-Ziv theorem for dihedral groups}, J. Pure Appl. Algebra 212 (2008), 311--319.

\bibitem{GG2013} A. Geroldinger and D. Grynkiewicz, \emph{The large Davenport constant I: groups with a cyclic, index $2$ subgroup}, J. Pure Appl. Algebra 217 (2013), 863--885.

\bibitem{GGOZ} A. Geroldinger, D. Grynkiewicz, J. Oh and Q. Zhong, \emph{On product-one sequences over dihedral groups}, J. Algebra Appl. https://www.worldscientific.com/doi/abs/10.1142.S0219498822500645 (2021).

\bibitem{Gr} D. Grynkiewicz, \emph{Structural Additive Theory}, Developments in Mathematics 30, Springer, Cham, 2013.

\bibitem{Ha} M. Hall, \emph{The theory of groups}, Reprinting of the 1968 edition, Chelsea Publishing Co, New York, 1976.

\bibitem{Han2015} D. Han, \emph{The Erd\H{o}s-Ginzburg-Ziv theorem for finite nilpotent groups}, Arch. Math. (Basel) 104 (2015), 325--332.

\bibitem{HZ2019} D. Han and H. Zhang, \emph{Erd\H{o}s-Ginzburg-Ziv theorem and Noether number for $C_m\ltimes_{\varphi} C_{mn}$}, J. Number Theory 198 (2019), 159--175.

\bibitem{OZ} J. Oh and Q. Zhong, \emph{On Erd\H{o}s-Ginzburg-Ziv inverse theorems for dihedral groups and dicyclic groups}, Isr. J. Math. 238 (2020), 715--743.

\bibitem{O} J. Olson, \emph{On a combinatorial problem of Erd\H{o}s, Ginzburg and Ziv}, J. Number Theory 8 (1976), 52--57.

\bibitem{QL} Y. Qu and Y. Li, \emph{On a conjecture of Zhuang and Gao}, \\
https://arxiv.org/abs/2107.06969.

\bibitem{YP1988} T. Yuster, \emph{Bounds for counter-exmple to an addition theorem in solvable groups}, Arch. Math. (Basel) 51 (1988), 223--231.

\bibitem{YP1984} T. Yuster and B. Peterson, \emph{A generalization of an addition theorem for solvable groups}, Canad. J. Math. 36 (1984), 529--536.

\bibitem{ZG2005} J. Zhuang and W. Gao, \emph{Erd\H{o}s-Ginzburg-Ziv theorem for dihedral groups of large prime index}, European J. Combin. 26 (2005), 1053--1059.

\end{thebibliography}
\end{document}